\documentclass[a4paper]{amsart}

\usepackage{color} 
\usepackage[all]{xy}

\usepackage{amsmath}
\usepackage{amssymb}
\usepackage{amsfonts}
\usepackage{amsthm}
\usepackage{amscd}
\usepackage{pifont}
\usepackage{latexsym}
\usepackage{url}

\usepackage{graphicx}

\usepackage{enumitem} 

\numberwithin{equation}{section}

\theoremstyle{plain}
\newtheorem{theorem}{Theorem}[section]
\newtheorem{lemma}[theorem]{Lemma}
\newtheorem{proposition}[theorem]{Proposition}
\newtheorem{corollary}[theorem]{Corollary}

\theoremstyle{definition}
\newtheorem{definition}[theorem]{Definition}

\newcommand{\yy}{\rightarrow}

\newcommand{\ladj}[1]{{#1}_{\ast}}
\newcommand{\radj}[1]{{#1}^{\ast}}

\newcommand{\ident}{\mathrm{Id}}

\newcommand{\atom}{\mathcal{A}}

\newcommand{\power}{\mathcal{P}}

\newcommand{\nsystem}{\mathcal{V}}

\newcommand{\iand}{\bigwedge}
\newcommand{\ior}{\bigvee}

\newcommand{\da}{\mbox{$\downarrow$}}
\newcommand{\ua}{\mbox{$\uparrow$}}
\newcommand{\drel}[1]{\mbox{$\downarrow$}_{#1}}
\newcommand{\urel}[1]{\mbox{$\uparrow$}_{#1}}

\newcommand{\thn}{\ \Rightarrow\ }
\newcommand{\eq}{\ \Leftrightarrow\ }

\newcommand{\rel}[1]{<_{#1}}

\newcommand{\kfr}{\mathbf{KFr}}
\newcommand{\nfr}{\mathbf{NFr}}

\newcommand{\mkf}{\mathbf{MRKF}}
\newcommand{\caba}{\mathbf{CABA}}
\newcommand{\cama}{\mathbf{CAMA}}

\newcommand{\mdl}{\Diamond}

\title[Duality for $\kappa$-additive complete atomic modal algebras]{Duality for $\kappa$-additive 
complete atomic  modal algebras}

\author[Y. Tanaka]{Yoshihito Tanaka}
\address{Department of Economics\\
Kyushu Sangyo University\\Fukuoka 813-0015\\JAPAN}
\email{ytanaka@ip.kyusan-u.ac.jp}



\begin{document}

\maketitle

\begin{abstract}
In this paper, we give a duality theorem between the 
category of 
$\kappa$-additive complete atomic modal algebras and 
the category of $\kappa$-downward directed multi-relational Kripke frames, 
for any cardinal number $\kappa$. 
Multi-relational Kripke frames are not Kripke frames for multi-modal logic, 
but frames for monomodal logics in which the modal operator 
$\mdl$ does not distribute over (possibly infinite) disjunction, in general. 
We first define homomorphisms of multi-relational Kripke frames, and then
show the equivalence between
the category of $\kappa$-downward directed multi-relational Kripke frames 
and 
the category $\kappa$-complete neighborhood frames, 
from which the duality theorem follows. 
We also present another direct proof 
of this duality based on the technique given by Minari. 
\end{abstract}

\section{Introduction}

It is proved by Thomason \cite{thm75} that 
the category of all completely additive complete atomic modal algebras
is dually equivalent to the category of all Kripke frames, 
where a modal algebra is said to be completely additive, if the modal operator $\mdl$ 
distributes over the joins of every subsets of the algebra. 
However, there are some modal logics which cannot be characterized by 
a class of completely additive modal algebras. 
For example, 
if we see the existential and universal quantifiers as 
infinite joins and meets, respectively, 
the Barcan formula 
$\forall x\Box\phi\supset\Box\forall x\phi$ corresponds to the complete additivity, 
but there exist predicate modal logics in which it is not derivable. 
Moreover, there exists a propositional normal modal logic which is incomplete 
with respect to any class of completely additive complete modal algebras 
\cite{hld-ltk19}.

Subsequently, Do\v{s}en \cite{dsn89} 
gives broad kinds of duality theorems for 
categories of modal algebras and 
neighborhood frames, 
including 
duality between 
the category of complete atomic modal algebras and 
the category of 
$\omega$-complete neighborhood frames 
(which are called full filter frames in \cite{dsn89})
and 
that between
the category of completely additive complete atomic modal algebras 
and 
the category of 
complete neighborhood frames
(which are called full hyperfilter frames in \cite{dsn89}), 
as well as 
equivalence between 
the category of 
complete neighborhood frames and 
the category of Kripke frames. 
However, it should be remarked that 
the category of neighborhood frames is not a generalization of 
the category of Kripke frames, in the following sense: 
For any Kripke frame $F=\langle W,R\rangle$, we can define the "underlying" 
neighborhood frame $U(F)=\langle W,\nsystem_{F} \rangle$, where,  
$$
\nsystem_{F}(x)=\{\{y\mid (x,y)\in R\}\},  
$$
for any $x\in W$.  
However, as we will see in Theorem~\ref{nfr-mkf}, $U$ does not define 
the forgetful functor.

In this paper, we give another duality theorem for the category of 
complete atomic modal algebras between the 
category of multi-relational Kripke frames. 
Multi-relational Kripke frames are not Kripke frames for multi-modal logic, 
but frames for monomodal logics in which the modal operator 
$\mdl$ does not distribute over (possibly infinite) disjunction, in general. 
For example, in deontic logic (see, e.g., \cite{gbl00,cld13}), 
\begin{equation}\label{mdldistland}
(\Box p\land\Box q)\supset\Box(p\land q)
\end{equation}
should not be derived,
as the formula $(\Box\phi\land\Box\neg\phi)\supset\Box\psi$, 
which means that "if there is any conflict of obligation, then everything is  obligatory"
(\cite{gbl00}, p.114)  
can be deduced from it,  
and in the least infinitary modal logic, it is proved that the 
countable extension of 
(\ref{mdldistland}) is not derivable \cite{tnkcut,mnr16}. 
Consequently, 
these logics are Kripke incomplete, but it is proved that 
deontic logic $\mathrm{P}$
is complete with respect to the class of
serial multi-relational Kripke frames  \cite{gbl00}, and 
the least
infinitary modal logic is complete 
with respect to the class of $\omega$-downward directed multi-relational 
Kripke frames \cite{mnr16}. 
In this paper, we first define homomorphisms of multi-relational Kripke frames 
so that 
the category of multi-relational Kripke frames is going to be
a generalization of the category of Kripke frames. 
Then we show that the category of $\kappa$-downward directed 
multi-relational Kripke frames are equivalent to the category 
of $\kappa$-complete neighborhood frames for every cardinal number $\kappa$, 
which is a generalization of Do\v{s}en's equivalence theorem 
between the category of Kripke frames and the category of 
complete neighborhood frames. 
From this equivalence, duality between the category of $\kappa$-additive complete atomic 
modal algebras and the category of $\kappa$-downward directed 
multi-relational Kripke frames follows. 
In addition, we give another proof for this duality for any regular cardinal $\kappa$. 
The basic technique of this proof is given by Minari \cite{mnr16}. 
He proved completeness theorem for the least infinitary modal logic with 
respect to $\omega$-downward directed multi-relational Kripke frames 
by constructing a multi-relational Kripke frame such that 
each binary relation is given in the same way as the canonical frame of 
a finite fragment of the Lindenbaum algebra. 
We show that Minari's technique works also for homomorphisms 
and can be extended for any regular cardinal $\kappa$.

\section{Preliminaries}

In this section, we fix notation and recall definitions and 
basic results. 
For the details,   
see, e.g., \cite{blc-rjk-vnm01,gvn-hlm09}.  

Let $W$ be a non-empty set and $R$ a binary relation on $W$. 
For any $w_{1}$ and $w_{2}$ in $W$, we write 
$w_{1}\rel{R}w_{2}$ if $(w_{1},w_{2})\in R$. 
For any $X\subseteq W$, 
$\urel{R}X$ and $\drel{R}X$ denote the subsets of $W$ defined by 
$$
\urel{R}X=\{w\in W\mid \exists x\in X(x\rel{R}w)\},\ \  
\drel{R}X=\{w\in W\mid \exists x\in X(w\rel{R}x)\},
$$
respectively. If $X$ is a singleton $\{w\}$, we write 
$\urel{R}w$ and $\drel{R}w$ for 
$\urel{R}X$ and $\drel{R}X$, respectively.  
If $R$ is a partial order $\leq$, we write $\ua$ and $\da$ for 
$\urel{\leq}$ and $\drel{\leq}$, respectively.

Let $f:A\yy B$ be a mapping from a set $A$ to a set $B$. 
For any set $X\subseteq A$ and $Y\subseteq B$, 
$f\left[X\right]$ and $f^{-1}\left[Y\right]$ denote the sets
$$
f\left[X\right]
=
\{f(x)\mid x\in X\},\ \ 
f^{-1}\left[Y\right]
=
\{x\in X\mid f(x)\in Y\},  
$$
respectively.

\begin{definition}\rm
A Boolean algebra $A$ is said to be {\em complete} if for any 
$X\subseteq A$,  
$\ior X$ and $\iand X$ exist in $A$. 
Let $A$ and $B$ be complete Boolean algebras. A mapping $f:A\yy B$ is called 
a {\em homomorphism of complete Boolean algebras} if $f$ is a homomorphism 
of Boolean algebras which satisfies
$$
f\left(\ior X\right)=\ior f\left[X\right],\ 
f\left(\iand X\right)=\iand f\left[X\right]
$$
for any $X\subseteq A$. 
\end{definition}

\begin{definition}\rm 
For any homomorphism $f:A\yy B$ of complete Boolean algebras, 
$\radj{f}$ and $\ladj{f}$ denote mappings from $B$ to $A$ which are defined by 
$$
\radj{f}(b)=\ior f^{-1}\left[\da b\right],\ \  
\ladj{f}(b)=\iand f^{-1}\left[\ua b\right],
$$
for any $b\in B$, respectively.  
\end{definition}

\begin{proposition}
Let $f:A\yy B$ be a homomorphism of complete Boolean algebras.  
For any $a\in A$ and $b\in B$, 
\begin{equation}\label{adjoint}
f(a)\leq b \eq a\leq\radj{f}(b),\ \ 
b\leq f(a)\eq \ladj{f}(b)\leq a.
\end{equation}
That is, $\radj{f}$ and $\ladj{f}$ are right and left adjoints of $f$, respectively. 
\end{proposition}

It follows from (\ref{adjoint}) that $\radj{f}$ and $\ladj{f}$ are 
order preserving mappings and 
\begin{equation}\label{monotone}
f\circ\radj{f}, \ \ladj{f}\circ f\leq\ident_{B},
\ \ \ 
\ident_{A}
\leq\radj{f}\circ f,\ f\circ\ladj{f}. 
\end{equation}

\begin{definition}\rm
Let $A$ be a Boolean algebra. 
A non-zero element 
$a\in A$ is called an  {\em atom}
if $0< x\leq a$ implies $x=a$. 
The set of all atoms of  $A$ is denoted by 
$\atom(A)$. 
A Boolean algebra $A$ is said to be {\em atomic} 
if every non-zero element $x\in A$ satisfies
$$
x=\ior_{a\in\atom(A),\ a\leq x}a. 
$$
We write $\caba$ for the category whose objects are all complete and 
atomic Boolean algebras 
and arrows are all homomorphisms of complete Boolean algebras. 
\end{definition}

\begin{proposition}\label{atom-property}
Let $A$ be a Boolean algebra and $0\not=a\in A$. 
Then the following conditions are equivalent:
\begin{enumerate}
\item \label{con:atom}
$a$ is an atom. 
\item \label{con:compjoinirr}
For any $X\subseteq A$, if $\ior X\in A$ and $a\leq \ior X$ then 
$a\leq x$ for some $x\in X$. 
\item \label{con:joinirr}
For any $x$ and $y$ in $A$, if $a\leq x\lor y$ then 
$a\leq x$ or $a\leq y$. 
\item \label{con:complete}
For any $x\in A$, 
$a\leq x$ or $a\leq -x$. 
\end{enumerate}
\end{proposition}

\begin{proposition}
Let $A$ and $B$ be complete atomic Boolean algebras and  
$f:A\yy B$ a homomorphism of complete Boolean algebras. 
If $b\in\atom(B)$, then $\ladj{f}(b)\in\atom(A)$. 
\end{proposition}

\begin{definition}\label{kripke}\rm
A {\em Kripke frame} is a pair $\langle W,R\rangle$, 
where $W$ is a non-empty set and $R$ is a binary relation on $W$. 
Let $F_{1}=\langle W_{1},R_{1}\rangle$ and  
$F_{2}=\langle W_{2},R_{2}\rangle$ be Kripke frames. 
A {\em homomorphism $f:F_{1}\yy F_{2}$ of Kripke frames} is 
a mapping from $W_{1}$ to $W_{2}$ which satisfies the following:
\begin{enumerate}
\item
for any $v$ and $w$ in $W_{1}$, 
if $v\rel{R_{1}}w$ then $f(v)\rel{R_{2}}f(w)$;

\item
for any $w\in W_{1}$ and $u\in W_{2}$,  
if $f(w)\rel{R_{2}}u$ then there exists 
$v\in W_{1}$ such that $w\rel{R_{1}}v$ and 
$f(v)=u$. 
\end{enumerate}
We write $\kfr$ for the category of all Kripke frames. 
\end{definition}

\section{The category of complete atomic modal algebras}

\begin{definition}\rm
An algebra $\langle A;\lor,\land,-,\mdl,0,1\rangle$ is called 
a {\em modal algebra}
if its reduct $\langle A;\lor,\land,-,0,1\rangle$ is a Boolean algebra 
and $\mdl$ is a unary operator
which satisfies 
$
\mdl 0=0
$
and
$$
\mdl x\lor \mdl y=\mdl(x\lor y)
$$
for any $x$ and $y$ in $A$. 
A modal algebra $A$ is said to be {\em complete} or {\em atomic} if its 
Boolean reduct is complete or atomic, respectively. 
Let $A$ and $B$ be modal algebras. A mapping $f:A\yy B$ is called 
a {\em homomorphism of modal algebras} if $f$ is a homomorphism 
of Boolean algebras which satisfies
$$
f(\mdl x)=\mdl f(x)
$$
for any $x\in A$. 
A homomorphism $f$ of modal algebras is called 
a {\em homomorphism of complete modal algebras} if it is a homomorphism 
of complete Boolean algebras. 
\end{definition}

\begin{definition}\label{def:additive}\rm
A complete modal algebra $A$ is said to be {\em completely additive}
if
\begin{equation}\label{eqbarcan}
\ior_{x\in X}\mdl x=\mdl\ior X
\end{equation}
holds for any $X\subseteq A$. 
Let $\kappa$ be a cardinal number. 
A complete modal algebra $A$ is said to be {\em $\kappa$-additive} 
if the equation (\ref{eqbarcan})
holds for any $X\subseteq A$ such that $|X|<\kappa$. 
\end{definition}

\begin{definition}\rm
The objects of the category $\cama_{\infty}$ are
all completely additive complete atomic modal algebras and 
the arrows of it are all homomorphisms 
of complete modal algebras between them. 
Let $\kappa$ be a cardinal number. 
The objects of the category $\cama_{\kappa}$ are
all $\kappa$-additive complete atomic modal algebras and 
the arrows of it are all homomorphisms 
of complete modal algebras between them. 
\end{definition}

\begin{theorem}{\rm (Thomason \cite{thm75})}. 
$\cama_{\infty}$ and $\kfr$ are dually equivalent. 
\end{theorem}

\begin{proof}
First, we define a functor $F:\cama_{\infty}\yy\kfr$.  
For any object $A$ of $\cama_{\infty}$, define $F(A)$ by
$$
F(A)=\langle \atom(A),R\rangle, 
$$
where, 
$$
a\rel{R}b\eq a\leq \mdl b
$$
for any $a$ and $b$ in $\atom(A)$, 
and for any 
arrow $f:A\yy B$ of $\cama_{\infty}$, define 
$F(f):F(B)\yy F(A)$ by
$$
F(f)(b)=\ladj{f}(b)
$$
for any $b\in\atom(B)$. 
Next, 
we define a functor $G:\kfr\yy\cama_{\infty}$.  
For any object $K=\langle W,R\rangle$ of $\kfr$, 
define $G(K)$ by 
$$
G(K)=
\langle
\power(W);\cup,\cap,W\setminus-,\mdl_{K},\emptyset,W
\rangle, 
$$
where
$$
\mdl_{K}X=\drel{R}X
$$
for any $X\subseteq W$,
and for any arrow $g$ from 
$K_{1}=\langle W_{1},R_{1}\rangle$ 
to
$K_{2}=\langle W_{2},R_{2}\rangle$ 
of $\kfr$, define $G(g):G(K_{2})\yy G(K_{1})$ by 
$$
G(g)(X)=g^{-1}[X]
$$
for any $X\in\power(W_{2})$. 
Then $F:\cama_{\infty}\yy\kfr$ and $G:\kfr\yy\cama_{\infty}$ are 
well-defined contravariant functors and
$$
\ident_{\cama_{\infty}}\cong G\circ F,\ \ 
\ident_{\kfr}\cong F\circ G.  
$$
\end{proof}

\section{The category of neighborhood frames}\label{section:nfr}

A {\em neighborhood frame} is a pair
$\langle C, \nsystem\rangle$, where 
$C$ is a non-empty set and 
$\nsystem$ is a mapping from $C$ to $\power(\power(C))$. 
A neighborhood frame $\langle C, \nsystem\rangle$ is said to {\em include the whole set}
if for any $c\in C$, $C\in\nsystem(c)$, and 
is said to be {\em upward closed}
if for any $c\in C$, $X\in\nsystem(c)$, and $Y\subseteq C$, 
if $X\subseteq Y$ then $Y\in\nsystem(c)$. 
A neighborhood frame $\langle C, \nsystem\rangle$ is said to be 
{\em complete}
if it includes the whole set, is upward closed, and 
for any $c\in C$ and non-empty subset $S$ of $\nsystem(c)$, 
\begin{equation}\label{completenfr}
S\subseteq\nsystem(c)\thn\bigcap S\in \nsystem(c). 
\end{equation}
Let $\kappa$ be a cardinal number. 
A neighborhood frame $\langle C, \nsystem\rangle$ 
is said to be {\em $\kappa$-complete}
if it includes the whole set, is upward closed, and 
(\ref{completenfr}) holds 
for any non-empty subset $S$ of $\nsystem(c)$ such that 
$|S|<\kappa$.

Let $Z_{1}=\langle C_{1},\nsystem_{1}\rangle$ and $Z_{2}=\langle C_{2},\nsystem_{2}\rangle$
be neighborhood frames. 
A mapping $f:C_{1}\yy C_{2}$ is called a 
{\em homomorphism of neighborhood frames} from $Z_{1}$ to $Z_{2}$ if for 
any $c\in C_{1}$ and $X\subseteq C_{2}$, 
$$
f^{-1}[X]\in \nsystem_{1}(c)\eq X\in \nsystem_{2}(f(c))
$$
holds. 

We write $\nfr$ for the category of all neighborhood frames. 
We also write $\nfr_{\infty}$ and $\nfr_{\kappa}$ for its full subcategories of all 
complete neighborhood frames and 
all $\kappa$-complete neighborhood frames, respectively.  
The duality theorem between $\nfr_{\omega}$ and $\cama$ and that
between $\nfr_{\infty}$ and $\cama_{\infty}$, which are 
given in Do\v{s}en \cite{dsn89}, can be generalized to 
any cardinal number $\kappa$, immediately:

\begin{theorem}\label{camanfr}
{\rm (Do\v{s}en \cite{dsn89})}. 
For any cardinal number $\kappa$, 
$\cama_{\kappa}$ and $\nfr_{\kappa}$ are dually equivalent. 
\end{theorem}

\begin{proof}
First, we define a functor $J:\cama_{\kappa}\yy\nfr_{\kappa}$. 
For any object $A$ of $\cama_{\kappa}$, define 
$J(A)$ by 
$$
J(A)=\langle \atom(A),\nsystem\rangle, 
$$
where 
$$
\nsystem(a)=\{\atom(A)\cap\da x\mid
a\not\leq\mdl-x\}
$$
for any $a$, 
and for any 
arrow $f:A\yy B$ of $\cama_{\kappa}$, define 
$J(f):J(B)\yy J(A)$ by 
$$
J(f)(b)=\ladj{f}(b)
$$
for any $b\in\atom(B)$. 
Next, 
we define a functor $K:\nfr_{\kappa}\yy\cama_{\kappa}$. 
For any object $Z=\langle C,\nsystem\rangle$ of $\nfr_{\kappa}$, 
define $K(Z)$ by 
$$
K(Z)=
\langle
\power(C);\cup,\cap,C\setminus-,\mdl_{Z},\emptyset,C
\rangle, 
$$
where
$$
\mdl_{Z}X=\{c\in C\mid C\setminus X\not\in\nsystem(c)\}
$$
for any $X\subseteq C$,
and for any arrow $g$ from 
$Z_{1}=\langle C_{1},\nsystem_{1}\rangle$ 
to
$Z_{2}=\langle C_{2},\nsystem_{2}\rangle$ 
of $\nfr_{\kappa}$, define $K(g):K(Z_{2})\yy K(Z_{1})$ by 
$$
K(g)(X)=g^{-1}[X]
$$
for any $X\in\power(C_{2})$. 
Then $J:\cama_{\kappa}\yy\nfr_{\kappa}$ and $K:\nfr_{\kappa}\yy\cama_{\kappa}$ are 
well-defined contravariant functors and 
$$
\delta:\ident_{\cama_{\kappa}}\cong K\circ J,\ \ \ 
\gamma:\ident_{\nfr_{\kappa}}\cong J\circ K,   
$$
where the natural isomorphisms $\delta$ and $\gamma$ are defined by 
$$
\delta_{A}:x\mapsto \{a\in\atom(A)\mid a\leq x\},
\ \ 
\gamma_{Z}:y\mapsto\{y\},  
$$
for any object $A$ in $\cama_{\kappa}$ and any $Z$ in $\nfr_{\kappa}$.  
\end{proof}

Do\v{s}en also proved the following equivalence of categories: 

\begin{theorem}\label{nfrkfr}
{\rm (Do\v{s}en \cite{dsn89})}. 
$\nfr_{\infty}\cong\kfr$. 
\end{theorem}

For any Kripke frame $F=\langle W, R\rangle$, we can define a neighborhood frame $U(F)$ by 
$U(F)=\langle W, \{\urel{R}x\mid x\in W\}\rangle$. 
However, as is shown in Theorem~\ref{nfr-mkf} there exists 
a Kripke frame $F$ such that $U(F)$ is not a complete neighborhood frame 
and there exists a homomorphism $f:F_{1}\yy F_{2}$ of Kripke frames which is not 
a homomorphism of neighborhood frames from $U(F_{1})$ to $U(F_{2})$. 
In this sense, the neighborhood frames are not a generalization of the Kripke frames, 
although the two categories are equivalent.

\section{The category of multi-relational Kripke frames}

\begin{definition}\label{def:mkf}\rm
A pair $\langle W,S\rangle$ is called a {\em multi-relational 
Kripke frame} if $W$ is a non-empty set and 
$S$ is a non-empty set of binary relations on $W$. 
A multi-relational Kripke frame $\langle W,S\rangle$ is said to 
be {\em completely downward directed} if for any $S'\subseteq S$, 
there exists $R\in S$ such that 
\begin{equation}\label{dd}
R\subseteq \bigcap S'. 
\end{equation}
Clearly, $\langle W,S\rangle$ is completely downward directed if and only if 
$\bigcap S \in S$. 
Let $\kappa$ be a cardinal number. 
A multi-relational Kripke frame $\langle W,S\rangle$ is said to 
be {\em $\kappa$-downward directed} if for any 
$S'\subseteq S$ such that $|S'|<\kappa$, there exists $R\in S$ 
which satisfies (\ref{dd}). 
Let $M_{1}=\langle W_{1},S_{1}\rangle$ and $M_{2}=\langle W_{2},S_{2}\rangle$ be 
multi-relational Kripke frames. 
A mapping $f:W_{1}\yy W_{2}$ is called a {\em homomorphism of multi-relational 
Kripke frames} from $M_{1}$ to $M_{2}$ if it satisfies the following two conditions:
\begin{enumerate} 
\item 
for any $x\in W_{1}$ and $R_{2}\in S_{2}$, there exists $R_{1}\in S_{1}$ such 
that for any $y\in W_{1}$,
\begin{equation*}
x\rel{R_{1}}y\thn f(x)\rel{R_{2}}f(y);
\end{equation*}

\item
for any $x\in W_{1}$ and $R_{1}\in S_{1}$, there exists $R_{2}\in S_{2}$ such 
that for any $u\in W_{2}$,
\begin{equation*}
f(x)\rel{R_{2}}u \thn \mbox{$\exists y\in W_{1}$ such that 
$x\rel{R_{1}}y$
and
$f(y)=u$}. 
\end{equation*}
\end{enumerate}
A homomorphism of multi-relational Kripke
frames is an {\em isomorphism} if it is bijective. 
Indeed, if $f$ is an isomorphism, its inverse is also a homomorphism of 
multi-relational Kripke frames. 
\end{definition}

\begin{definition}
We write $\mkf$ for the category of all multi-relational Kripke frames. 
We also write 
$\mkf_{\infty}$ 
and 
$\mkf_{\kappa}$ 
for its full subcategories of all completely downward directed 
multi-relational Kripke frames and 
all $\kappa$-downward directed multi-relational Kripke frames,
respectively.  
\end{definition}

The following theorem states that the
multi-relational Kripke frames can be seen as a generalization of 
the Kripke frames:  

\begin{proposition}\label{kfrmkf}
For any Kripke frame $F=\langle W,R\rangle$, define $M(F)$ by 
$
M(F)=
\langle W,\{R\}\rangle,
$
and for any homomorphism $f:F_{1}\yy F_{2}$ of Kripke frames, define $M(f)$ by $f$. 
Then, $M$ is a well-defined functor and the image of $\kfr$ by $M$
is a full and faithful subcategory of $\mkf_{\infty}$. 
\end{proposition}

\begin{proof}
Clear from the definition of the homomorphism of multi-relational Kripke frames. 
\end{proof}

It is easy to prove the equivalence of $\mkf_{\infty}$ and $\kfr$.  
Define a functor $L:\mkf_{\infty}\yy\kfr$ 
by
$$
L:\langle W,S\rangle\mapsto
\left\langle W,\bigcap S\right\rangle,\ \ 
L(f)=f. 
$$
Then it is easy to show that $L$ is a well-defined functor 
and both $L\circ M\cong\ident_{\kfr}$ and $M\circ L\cong\ident_{\mkf_{\infty}}$ hold.

\section{Equivalence between $\mkf$ and $\nfr$}\label{section:enfr}

For any multi-relational Kripke frame $M=\langle W,S\rangle$, 
we can define the "underlying"  neighborhood frame $U(M)$  by 
$U(M)=\langle W,\nsystem_{M}\rangle$, where 
$$
\nsystem_{M}(x)=\{\urel{R} x\mid R\in S\}. 
$$ 
However, $U$ does not define the forgetful functor 
from $\mkf$ to $\nfr$ nor 
that from $\mkf_{\kappa}$ to $\nfr_{\kappa}$. 
In fact, we have the following:

\begin{theorem}\label{nfr-mkf}
\begin{enumerate}
\item 
There exists an object $M$ of $\mkf_{\kappa}$ such that $U(M)$ is not an 
object of $\nfr_{\kappa}$. 
Moreover, there exists such an object $M$ in $\kfr$
such that $U(M)$ is not an object of $\nfr_{\infty}$. 
\item 
There exists an arrow $f:M_{1}\yy M_{2}$ of $\mkf_{\kappa}$ such that $U(M_{1})$ and 
$U(M_{2})$ are objects of $\nfr$ but $f$ is not an arrow of $\nfr$. 
Moreover, there exists such an arrow $f$ in $\kfr$, either. 
\item
There exists an arrow $f:U(M_{1})\yy U(M_{2})$ of $\nfr$ such that 
$M_{1}$ and $M_{2}$ are objects of $\mkf$ but 
$f:M_{1}\yy M_{2}$ is not an arrow of $\mkf$. 
\end{enumerate}
\end{theorem}

\begin{proof}
\noindent
(1): 
Let $M=\left\langle \{0\},\{\emptyset\}\right\rangle$. 
Then $M$ is an object of $\mkf_{\kappa}$, but not that of $\nfr_{\kappa}$, 
since $\nsystem_{M}(0)$ is not upward closed. 
If we identify a singleton $\{R\}$ of a relation with $R$, $M$ is a Kripke frame, either. 

\noindent
(2): 
Let
$M_{1}=\left\langle \{0\},\{\{(0,0)\}\}\right\rangle$
and
$M_{2}=\left\langle \{0,1\},\{\{(0,0)\}\}\right\rangle$. 
Let $f:0\mapsto 0$. 
It is easy to see that 
$f\in\hom_{\mkf_{\kappa}}(M_{1},M_{2})$.  
If we identify a singleton $\{R\}$ of a relation with $R$, $f$ is a homomorphism 
of Kripke frames, either. 
However, $f$ is not an arrow of $\nfr$ from $U(M_{1})$ to $U(M_{2})$, since 
$
f^{-1}\left[\{0,1\}\right]
=
\{0\}\in\nsystem_{M_{1}}(0)$, 
but $\{0,1\}\not\in\nsystem_{M_{2}}(0)$.

\noindent
(3): 
Let
$
M_{1}=\left\langle \{0,1,2\}, \{R_{1},R_{2}\}\right\rangle$ and 
$M_{2}=\left\langle \{0,1\}, \{Q\}\right\rangle,
$ 
where 
$$
R_{1}=\{(0,1)\},\ \ 
R_{2}=\{(0,0),(0,1),(0,2)\},\ \ 
Q=\{(0,0),(0,1)\}. 
$$
Then 
$$
\nsystem_{M_{1}}(0)=\{\{1\},\{0,1,2\}\},\ \
\nsystem_{M_{1}}(1)=\nsystem_{M_{1}}(2)=\{\emptyset\}
$$
and
$$\nsystem_{M_{2}}(0)=\{\{0,1\}\},\ \ 
\nsystem_{M_{2}}(1)=\{\emptyset\}. 
$$
Define $f:\{0,1,2\}\yy\{0,1\}$ by $f(0)=0$ and $f(1)=f(2)=1$. 
It is easy to see that $f\in \hom_{\nfr}(U(M_{1}),U(M_{2}))$.  
However, $f\not\in \hom_{\mkf}(M_{1},M_{2})$, since $0\rel{Q}0$ but 
$0\not\rel{R_{1}}0$.   
\end{proof}

If we identify $U(M)$ with $M$ and 
a singleton $\{R\}$ of a relation with $R$, 
Proposition~\ref{kfrmkf} and 
Theorem~\ref{nfr-mkf} can be summarized as follows: 

\begin{equation*}
\begin{array}{ccc}
\nfr
&
\mathop{{\not\supseteq}\atop{\not\subseteq}}\limits_{\mathrm{arrows}}
&
\mkf
\\
\mathop{\rotatebox[origin=c]{90}{$\subseteq$}}\limits_{\mathrm{\ }}
&
&
\mathop{\rotatebox[origin=c]{90}{$\subseteq$}}\limits_{\mathrm{\ }}
\\
\nfr_{\kappa}
&
\mathop{\subsetneqq}\limits_{\mathrm{objects}}
&
\mkf_{\kappa}
\\
\mathop{\rotatebox[origin=c]{90}{$\subseteq$}}\limits_{\mathrm{\ }}
&
&
\mathop{\rotatebox[origin=c]{90}{$\subseteq$}}\limits_{\mathrm{\ }}
\\
\nfr_{\infty}
&
\mathop{\subsetneqq}\limits_{\mathrm{objects}}
&
\kfr
\end{array}
\end{equation*}

In the rest of this section, we show that $\nfr_{\kappa}$ and $\mkf_{\kappa}$ 
are equivalent. First, we show the following lemmas:

\begin{lemma}\label{mkftonfr}
Let $\kappa$ be any cardinal number. 
For any $\kappa$-downward directed multi-relational Kripke frame $M=\langle W, S\rangle$, 
define a $\kappa$-complete neighborhood frame $N(M)$ by 
$
\langle W,\nsystem_{M}\rangle 
$,
where
$\nsystem_{M}\subseteq\power(W)$ is defined by 
$$
\nsystem_{M}=
\ua\{\urel{R}x\mid R\in S\} 
$$
for any $x\in W$, 
and for any homomorphism $f$ of multi-relational Kripke frames, 
define $N(f)$ by $f$. 
Then $N$ is a full functor from $\mkf_{\kappa}$ to $\nfr_{\kappa}$.
\end{lemma}

\begin{proof}
It is clear that $N(M)$ is an  object of $\nfr_{\kappa}$. 
We show that 
for any $M_{1}=\langle W_{1},S_{1}\rangle$ 
and 
$M_{2}=\langle W_{2},S_{2}\rangle$,  
$$
\hom_{\mkf_{\kappa}}(M_{1},M_{2})=\hom_{\nfr_{\kappa}}(N(M_{1}),N(M_{2})).  
$$
\noindent
($\subseteq$): 
Suppose $f\in\hom_{\mkf_{\kappa}}(M_{1},M_{2})$. Take any $x\in W_{1}$ and $Y\subseteq W_{2}$. 
By definition of $\nsystem_{M_{1}}(x)$ and $\nsystem_{M_{2}}(x)$, 
\begin{eqnarray*}
Y\in\nsystem_{M_{2}}(f(x))
&\eq&
\exists Q\in S_{2}\left(\urel{Q}f(x)\subseteq Y\right)\\
&\thn&
\exists R\in S_{1}\left(f\left[\urel{R} x\right]\subseteq \urel{Q}f(x)\subseteq Y\right)\\
&\thn&
\exists R\in S_{1}\left(\urel{R} x\subseteq f^{-1}\left[Y\right]\right)\\
&\eq&
f^{-1}\left[Y\right]\in\nsystem_{M_{1}}(x).
\end{eqnarray*}
Conversely, 
\begin{eqnarray*}
f^{-1}\left[Y\right]\in\nsystem_{M_{1}}(x)
&\eq&
\exists R\in S_{1}\left(\urel{R}x\subseteq f^{-1}\left[Y\right]\right)\\
&\thn&
\exists Q\in S_{2}\left(\urel{Q}f(x)\subseteq f\left[\urel{R}x\right]\subseteq
f\left[f^{-1}\left[Y\right]\right]\right)\\
&\thn&
\exists Q\in S_{2}\left(\urel{Q}f(x)\subseteq Y\right)\\
&\eq&
Y\in\nsystem_{M_{2}}(f(x)). 
\end{eqnarray*}

\noindent
($\supseteq$):
Suppose $f\in\hom_{\nfr_{\kappa}}(N(M_{1}),N(M_{2}))$ and $x\in W_{1}$. 
First, take any $Q\in S_{2}$.
Then 
$
f^{-1}\left[\urel{Q}f(x)\right]$ is in $\nsystem_{M_{1}}(x)$, 
since $\urel{Q}f(x)\in \nsystem_{M_{2}}(f(x))$.  
Hence, there exists $R\in S_{1}$ such that 
$$
\urel{R}x\subseteq f^{-1}\left[\urel{Q}f(x)\right]. 
$$
Then for any $y\in W_{1}$,  if $x\rel{R} y$ then $f(x)\rel{Q} f(y)$.  
Next, take any $R\in S_{1}$. 
Since $\urel{R}x\in\nsystem_{M_{1}}(x)$ and $\nsystem_{M_{1}}(x)$ is upward closed, 
$$
\urel{R}x\subseteq
f^{-1}\left[f\left[\urel{R}x\right]\right]\in\nsystem_{M_{1}}(x). 
$$
Hence, $f\left[\urel{R}x\right]\in\nsystem_{M_{2}}(f(x))$. 
Then there exists $Q\in S_{2}$ such that 
$\urel{Q}f(x)\subseteq f\left[\urel{R}x\right]$. 
Then for any $u\in W_{2}$ such that $f(x)\rel{Q}u$, 
there exists $y\in W_{2}$ such that $x\rel{R}y$ and $f(y)=u$. 
\end{proof}

\begin{lemma}\label{nfrtomkf}
Let $Z=\langle C,\nsystem\rangle$ be any $\kappa$-complete neighborhood frame.  
We write $V_{Z}$ for the set 
$$
V_{Z}=\{v:C\yy\power(C)\mid \forall x\in C(v(x)\in\nsystem(x))\},  
$$
and for any $v\in V_{Z}$, we write $R_{v}$ for a binary relation on $C$ 
defined by 
$$
R_{v}=\{(x,y)\mid x\in C,\ y\in v(x)\}.  
$$
Then 
$
H(Z)=
\langle C,S_{Z}\rangle, 
$
is a 
$\kappa$-downward directed multi-relational Kripke frame, 
where 
$
S_{Z}=\{R_{v}\mid v\in V_{Z}\}$.  
If we define $H(f)$ by $f$ 
for any homomorphism $f$ of neighborhood frames, 
then $H$ is a functor from $\nfr_{\kappa}$ to $\mkf_{\kappa}$.
\end{lemma}

\begin{proof}
We first show that $H(Z)$ is an object of $\mkf_{\kappa}$. 
Since $Z$ includes the whole set,  $V_{Z}\not=\emptyset$. 
Therefore, $S_{Z}\not=\emptyset$. 
Take any subset $\{R_{v_{i}}\mid i\in\kappa\}$ of $S_{Z}$. As 
$Z$ is $\kappa$-complete, 
there exists $u\in V_{Z}$ such that 
$$
u(x)=\bigcap_{i\in\kappa}v_{i}(x)\in\nsystem(x)
$$
for any $x\in C$. 
Then $R_{u}=\bigcap_{i\in\kappa}R_{v_{i}}$. 
Next, we show that  
$$
\hom_{\nfr_{\kappa}}(Z_{1}, Z_{2})\subseteq 
\hom_{\mkf_{\kappa}}(H(Z_{1}), H(Z_{2}))
$$ 
for any $\kappa$-complete neighborhood frames 
$Z_{1}=\langle C_{1},\nsystem_{1}\rangle$ and $Z_{2}=\langle C_{2},\nsystem_{2}\rangle$.
Suppose $f\in \hom_{\nfr_{\kappa}}(Z_{1}, Z_{2})$. First, take any $x\in C_{1}$ and 
$R_{v}\in S_{Z_{2}}$. Then $\urel{R_{v}}f(x)=v(f(x))\in\nsystem_{2}(f(x))$. 
Hence, $f^{-1}\left[\urel{R_{v}}f(x)\right]\in \nsystem_{1}(x)$. 
By definition of $V_{Z_{1}}$, there exists $u\in V_{Z_{1}}$ such that 
$$
\urel{R_{u}}x=
u(x)=f^{-1}\left[\urel{R_{v}}f(x)\right]. 
$$
Hence, for any $y\in C_{1}$, 
$$
x\rel{R_{u}}y\eq y\in u(x)\eq f(x)\rel{R_{v}} f(y).
$$
Next, 
take any $x\in C_{1}$ and 
$R_{u}\in S_{Z_{1}}$. Then $\urel{R_{u}}x=u(x)\in\nsystem_{1}(x)$. 
Since $Z_{1}$ is upward closed, 
$f^{-1}\left[f\left[\urel{R_{u}}x\right]\right]\in\nsystem_{1}(x)$. 
Therefore, $f\left[\urel{R_{u}}x\right]\in\nsystem_{2}(f(x))$. 
By definition of $V_{Z_{2}}$, there exists $v\in V_{Z_{2}}$ such that 
$$
\urel{R_{v}}f(x)=
v(f(x))=f\left[\urel{R_{u}}x\right]. 
$$
Hence, for any $z\in C_{2}$ such that
$
f(x)\rel{R_{v}}z
$, 
there exists $y\in C_{1}$ such that $x\rel{R_{u}}y$ and $f(y)=z$. 

\end{proof}

Now, we prove that $\mkf_{\kappa}$ and $\nfr_{\kappa}$ are equivalent, 
which is a generalization of 
Theorem~\ref{nfrkfr}:

\begin{theorem}\label{mkfnfr}
$N$ and $H$ are equivalence between 
$\mkf_{\kappa}$ and $\nfr_{\kappa}$, for every cardinal number $\kappa$. 
\end{theorem}

\begin{proof}
For any object $M=\langle W,S\rangle$ of $\mkf_{\kappa}$, define a map 
$
\gamma_{M}:M\yy H(N(M))
$
by $\gamma_{M}(x)=x$ for any $x\in W$, and 
for any object $Z=\langle C,\nsystem\rangle$ of $\nfr_{\kappa}$, define a map 
$
\delta_{Z}:Z\yy H(N(Z))
$
by $\delta_{Z}(c)=c$ for any $c\in C$. 
It is trivial that 
$
H(N(f))\circ\gamma_{M_{1}}=\gamma_{M_{2}}\circ f
$
holds for any $f:M_{1}\yy M_{2}$
and 
$
N(H(g))\circ\delta_{Z_{1}}=\delta_{Z_{2}}\circ g
$
holds for any $g:Z_{1}\yy Z_{2}$. 

First, we show that for any multi-relational Kripke frame $M=\langle W,S\rangle$, 
$\gamma_{M}$ is an isomorphism of multi-relational Kripke frames from $M$ to $H(N(M))$.  
We check the first condition of the homomorphisms of multi-relational Kripke frames:
Take any $x\in W$ and $R_{v}\in S_{N(M)}$, where $v\in V_{N(M)}(x)$. 
Then there exists $R\in S$ such that 
$
\urel{R}x\subseteq v(x) 
$. 
For any $y\in W$, $R$ satisfies that 
\begin{equation*}\label{naturaliso}
x\rel{R}y
\thn
y\in v(x)
\eq
x\rel{R_{v}}y. 
\end{equation*}
Then we check the second condition: Take any $x\in W$ and $R\in S$. 
As $\urel{R}x\in\nsystem_{M}(x)$, there exists $v\in V_{N(M)}$ such that 
$\urel{R}x=v(x)$. 
Then $R_{v}\in S_{N(M)}$ satisfies that for any $y\in W$,
$$
x\rel{R_{v}}y
\eq
y\in v(x)
\eq
x\rel{R}y.
$$
As $\gamma_{M}$ is the identity mapping on $W$, 
$\gamma_{M}$ is an isomorphism of multi-relational Kripke frames.

Next, we prove that for any neighborhood frame $Z=\langle C,\nsystem\rangle$, 
$\delta_{Z}$ is an isomorphism of neighborhood frames from $Z$ to $N(H(Z))$.  
Since $Z$ is upward closed, 
\begin{equation}\label{upwardclosed}
\ua\{\urel{R_{v}}x\mid v\in V_{Z}\}
=
\{\urel{R}x\mid v\in V_{Z}\}
\end{equation}
for any $x\in C$. 
Take any $c\in C$ and $X\subseteq C$. Then, 
\begin{align*}
X\in\nsystem(x)
&\eq
\exists v\in V_{Z}\left(v(x)= X\right)\\
&\eq
\exists R_{v}\in S_{Z}\left(\urel{R_{v}}x= X\right)\\
&\eq
X\in\nsystem_{H(Z)}(x) &&(\text{by \ref{upwardclosed}}).
\end{align*}
As $\delta_{Z}$ is 
the identity mapping on $C$, 
$\delta_{Z}$ is an isomorphism of neighborhood frames. 
\end{proof}

By Theorem~\ref{camanfr} and Theorem~\ref{mkfnfr}, 
we have the following: 

\begin{theorem}
For any cardinal number $\kappa$, 
$\cama_{\kappa}$ and $\mkf_{\kappa}$ are dually equivalent. 
\end{theorem}

By the same argument as Theorem~\ref{mkfnfr}, 
it follows that the category of all multi-relational Kripke 
frames are equivalent to the category of 
all upward closed neighborhood frames which includes the whole set. 
These categories are dually equivalent to the category of 
algebras which is obtained by weakening the definition of 
the modal operator in $\cama$ to the following; $\mdl 0=0$ and 
$\mdl x\leq \mdl y$ whenever $x\leq y$.

\section{Functor from $\cama_{\kappa}$ to $\mkf_{\kappa}$}

In the rest of the paper, we give another direct proof of duality 
between $\cama_{\kappa}$ and $\mkf_{\kappa}$
for every regular cardinal $\kappa$. 
First, 
we define a contravariant functor $F:\cama_{\kappa}\yy\mkf_{\kappa}$ 
for every regular cardinal $\kappa$. 
For any object $A$ of $\cama_{\kappa}$, 
a multi-relational Kripke frame 
$F(A)$ is defined by
$$
F(A)=\langle\atom(A),\{R(X)\mid X\subseteq A,\ |X|<\kappa\}\rangle,
$$
where, for any $a\in\atom(A)$ and $b\in\atom(B)$,
$$
a\rel{R(X)}b
\eq
a\leq\iand \mdl\left[\ua b\cap X\right],
$$ 
and 
for any arrow $f:A\yy B$ of $\cama_{\kappa}$, 
the mapping $F(f):\atom(B)\yy\atom(A)$ is defined by 
$$
F(f)(b)=\ladj{f}(b) 
$$
for any $b\in\atom(B)$. 
Below, we show that $F$ is a well-defined contravariant functor.

\begin{proposition}
Let $\kappa$ be a regular cardinal.  
If $A$ is a $\kappa$-additive complete atomic modal algebra, 
$F(A)$ is a $\kappa$-downward directed multi-relational Kripke frame. 
\end{proposition}

\begin{proof}
It is clear that $F(A)$ is a multi-relational Kripke frame. 
We show that $F(A)$ is $\kappa$-downward directed. 
Suppose $X_{i}\subseteq A$ and $|X_{i}|<\kappa$ for any $i\in I$.  
If $|I|<\kappa$,  then
$$
|\bigcup_{i\in I}X_{i}|<\kappa,
$$
since $\kappa$ is regular. 
Hence, $F(A)$ is $\kappa$-downward directed, because 
$$
R\left(\bigcup_{i\in I}X_{i}\right)\subseteq \bigcap_{i\in I} R(X_{i}). 
$$
\end{proof}

\begin{definition}\rm
Let $A$ be a $\kappa$-additive complete atomic modal algebra.    
For any $X\subseteq A$ and $a\in \atom(A)$,  
$p(X,a)$ denotes an element of $A$ defined by 
$$
p(X,a)=\ior\mdl^{-1}\left[\da(-a)\right]\cap X.   
$$
\end{definition}

\begin{lemma}\label{equivalence}
Let $A$ be a $\kappa$-additive complete atomic modal algebra,   
$X$ a subset of $A$ such that $|X|<\kappa$, and $a\in\atom(A)$. 
Then 
for any $a'\in\atom(A)$, 
$$
a\rel{R(X)}a'
\eq
a'\not\leq p(X,a). 
$$
\end{lemma}

\begin{proof}
For any $a'\in\atom(A)$, 
\begin{align*}
a\rel{R(X)}a'
&\eq
a\leq \iand\mdl\left[\ua a'\cap X\right]\\
&\eq
\forall x\in X
(
a'\leq x\thn a\leq \mdl x
)\\
&\eq
\forall x\in X
(
a\not\leq \mdl x
\thn
a'\not\leq x
)\\
&\eq
\forall x\in X
(
a\leq -\mdl x
\thn
a'\not\leq x) 
&& \text{($a\in\atom(A)$)}
\\
&\eq
\forall x\in X
(
\mdl x\leq -a
\thn
a'\not\leq x
)\\
&\eq
\forall x
\left(
x\in\mdl^{-1}\left[\da(-a)\right]\cap X
\thn
a'\not\leq x
\right)\\
&\eq
a'
\not\leq
\ior\mdl^{-1}\left[\da(-a)\right]\cap X && \text{($a'\in\atom(A)$)}
.
\end{align*}
\end{proof}

\begin{lemma}\label{inv-equiv}
Let
$A$ and $B$ be $\kappa$-additive complete atomic modal algebras, 
$f:A\yy B$ a homomorphism of complete modal algebras,  
$Y\subseteq B$ such that $|Y|<\kappa$, and $b\in\atom(B)$.  
Suppose $X=\{\radj{f}(p(Y,b))\}$. 
Then for any $a\in\atom(A)$, 
$$
\ladj{f}(b)\rel{R(X)}a
\eq
a\not\leq\radj{f}(p(Y,b)). 
$$
\end{lemma}

\begin{proof}
By Lemma \ref{equivalence}, all we have to prove is
$$
\radj{f}(p(Y,b))=p(X,\ladj{f}(b)).    
$$
As 
$$
p(X,\ladj{f}(b))    
=
\ior\mdl^{-1}\left[\da(-\ladj{f}(b))\right]\cap \{\radj{f}(p(Y,b))\},    
$$  
it is enough to show
$$
\radj{f}(p(Y,b))\in \mdl^{-1}\left[\da\left(-\ladj{f}(b)\right)\right]. 
$$
Since $B$ is $\kappa$-additive
\begin{align*}
\mdl f(\radj{f}(p(Y,b)))
&\leq
\mdl p(Y,b)
&& \text{(by (\ref{monotone}))}\\
&=
\mdl\ior\mdl^{-1}\left[\da(-b)\right]\cap Y  \\
&=
\ior\mdl\left(\mdl^{-1}\left[\da(-b)\right]\cap Y\right)  &&
\text{($\kappa$-additivity)}\\
&\leq 
\ior\da(-b)\\
&=
-b. 
\end{align*}
Hence
$$
b\leq-\mdl f(\radj{f}(p(Y,b)))=f(-\mdl\radj{f}(p(Y,b))).   
$$
By (\ref{adjoint}),
$$
\ladj{f}(b)\leq-\mdl\radj{f}(p(Y,b)),  
$$
so
\begin{equation*}\label{property}
\mdl\radj{f}(p(Y,b))\leq-\ladj{f}(b). 
\end{equation*}
Hence, 
$$
\radj{f}(p(Y,b))\in \mdl^{-1}\left[\da\left(-\ladj{f}(b)\right)\right]. 
$$
\end{proof}

\begin{proposition}
Let $\kappa$ be a regular cardinal. 
For any $\kappa$-additive complete atomic modal algebras $A$ and $B$ and 
for any homomorphism $f:A\yy B$  of complete modal algebras, 
$F(f):\atom(B)\yy\atom(A)$ 
is a homomorphism of multi-relational Kripke frames from $F(B)$ to $F(A)$. 
\end{proposition}

\begin{proof}
Condition 1 of Definition \ref{def:mkf}: 
Take any $b_{1}\in\atom(B)$ and any $X\subseteq A$ such that $|X|<\kappa$. 
Then $|f\left[X\right]|<\kappa$. Take any $b_{2}\in\atom(B)$. 
We show that 
$$
b_{1}\rel{R(f\left[X\right])}b_{2}
\thn
\ladj{f}(b_{1})\rel{R(X)}\ladj{f}(b_{2}).  
$$
Suppose $b_{1}\rel{R(f\left[X\right])}b_{2}$. 
By definition of $R(f\left[X\right])$, 
\begin{equation*}
b_{1}
\leq
\iand\mdl\left[\ua b_{2}\cap f\left[X\right]\right]. 
\end{equation*}
Therefore, 
\begin{equation}\label{eq:assump}
\ladj{f}(b_{1})
=
\iand_{x\in A,\ b_{1}\leq f(x)}x
\leq
\iand
\left\{
x\in A\mid \iand\mdl\left[\ua b_{2}\cap f\left[X\right]\right]\leq f(x)
\right\}. 
\end{equation}
On the other hand, 
\begin{equation}\label{eq:member}
\mdl\left[\ua\ladj{f}(b_{2})\cap X\right]
\subseteq
\left\{
x\in A\mid \iand\mdl\left[\ua b_{2}\cap f\left[X\right]\right]\leq f(x)
\right\}, 
\end{equation}
because, for any $z\in\mdl\left[\ua \ladj{f}(b_{2})\cap X\right]$, there exists 
$u\in X$ such that 
$$
\ladj{f}(b_{2})\leq u,\ \mdl u=z, 
$$
and, this implies $b_{2}\leq f(u)$ and $f(u)\in f\left[X\right]$, and therefore, 
$$
\iand\mdl\left[\ua b_{2}\cap f\left[X\right]\right]
\leq
\mdl f(u)
=
f(\mdl u)
=
f(z).
$$
By (\ref{eq:assump}) and (\ref{eq:member}), 
$$
\ladj{f}(b_{1})\leq \iand\mdl\left[\ua \ladj{f}(b_{2})\cap X\right]. 
$$
Hence, 
$$
\ladj{f}(b_{2})
\rel{R(X)}
\ladj{f}(b_{1}).
$$

\noindent
Condition 2 of Definition \ref{def:mkf}: 
Take any $b\in\atom(B)$ and any $Y\subseteq B$ such that 
$|Y|<\kappa$. 
Define $X\subseteq A$ by 
$$
X=\{\radj{f}(p(Y,b))\}. 
$$
Suppose $a\in\atom(A)$ and $\ladj{f}(b)\rel{R(X)} a$. 
Then 
$a\not\leq \radj{f}(p(Y,b))$
by Lemma \ref{inv-equiv}.  Hence,  $f(a)\not\leq p(Y,b)$. 
Since $B$ is atomic, there exists $b'\in\atom(B)$  such that 
$$
b'\leq f(a),\ \  b'\not\leq p(Y,b).
$$
Then $\ladj{f}(b')\leq a$, and  $b\rel{R(Y)}b'$ by  Lemma \ref{equivalence}. 
Since $\ladj{f}(b')$ and $a$ are in $\atom(A)$, 
$\ladj{f}(b')= a$.
\end{proof}

\section{Functor from $\mkf_{\kappa}$ to $\cama_{\kappa}$}

We define a contravariant functor $G:\mkf_{\kappa}\yy\cama_{\kappa}$ for 
every cardinal number $\kappa$. 
For any object $M=\langle W, S\rangle$ of $\mkf_{\kappa}$, 
a complete atomic modal algebra $G(M)$ is defined by 
$$
G(M)=
\langle
\power(W);\cup,\cap,W\setminus-,\mdl_{M},\emptyset,W
\rangle, 
$$
where $\mdl_{M}$ is defined by 
$$
\mdl_{M}X=\bigcap_{R\in S}\drel{R}X  
$$
for any $X\subseteq W$, 
and for any multi-relational Kripke frames 
$M_{1}=\langle W_{1},S_{1}\rangle$, 
$M_{2}=\langle W_{2},S_{2}\rangle$, and 
any arrow $g:M_{1}\yy M_{2}$ of $\mkf_{\kappa}$, 
the mapping $G(g):\power(W_{2})\yy\power(W_{1})$ is defined by 
$$
G(g)(X)=g^{-1}\left[X\right]
$$
for any $X\subseteq W_{2}$. 
Below, we show that $G$ is a well-defined contravariant functor.

\begin{proposition}
Let $\kappa$ be a cardinal number. 
If $M=\langle W, S\rangle$ is a 
$\kappa$-downward directed multi-relational 
Kripke frame, $G(g)(M)$ is a $\kappa$-additive 
complete atomic modal algebra. 
\end{proposition}

\begin{proof}
It is clear that 
$\langle
\power(W);\cup,\cap,W\setminus-,\emptyset
\rangle$
is a complete atomic Boolean algebra. 
Since $\drel{R}\emptyset=\emptyset$ for any $R\in S$, 
$$
\mdl_{M}\emptyset=\bigcap_{R\in S}\drel{R}\emptyset=\emptyset. 
$$
Let $\{X_{i}\}_{i\in I}$ be a subset of $\power(W)$ such that $|I|<\kappa$. 
Since $\mdl_{M}$ is order preserving, 
$$
\bigcup_{i\in I}\mdl_{M}X_{i}\subseteq \mdl_{M}\bigcup_{i\in I}X_{i}.  
$$
We show the converse. 
For any $w\in W$, 
\begin{align*}
w\not\in \bigcup_{i\in I}\mdl_{M}X_{i}
&\eq
w\not\in
\bigcup_{i\in I}\bigcap_{R\in S}\drel{R}X_{i}\\
&\eq
\forall{i\in I}\left(
w\not\in
\bigcap_{R\in S}\drel{R}X_{i}
\right)\\
&\eq
\forall{i\in I}\exists{R_{i}\in S}
\forall x\in X_{i}
\left(
w\not\rel{R_{i}} x
\right). 
\end{align*}
Since $M$ is $\kappa$-downward directed, there exists $Q\in S$ such that 
$$
Q\subseteq \bigcap_{i\in I} R_{i}. 
$$
Then 
$$
\forall{i\in I}
\forall x\in X_{i}
\left(
w\not\rel{Q} x
\right).   
$$
Thus, 
$$
w\not\in
\drel{Q}
\bigcup_{i\in I}
X_{i}. 
$$
Hence, 
$$
w\not\in
\bigcap_{R\in S}
\drel{R}
\bigcup_{i\in I}
X_{i}
=
\mdl_{M} \bigcup_{i\in I}
X_{i}.
$$
\end{proof}

\begin{proposition}\label{garrow}
Let $\kappa$ be a cardinal number. 
For any $\kappa$-downward directed multi-relational 
Kripke frames $M_{1}=\langle W_{1},S_{1}\rangle$, 
$M_{2}=\langle W_{2},S_{2}\rangle$
and 
a homomorphism $g:M_{1}\yy M_{2}$  of multi-relational 
Kripke frames, 
$G(g):\power(W_{2})\yy\power(W_{1})$ 
is a homomorphism of complete modal algebras from $G(M_{1})$ to $G(M_{2})$. 
\end{proposition}

\begin{proof}
We only show that for any $U\subseteq W_{2}$, 
$$
\mdl_{M_{1}}G(g)(U)=G(g)(\mdl_{M_{2}}U).  
$$
All we have to prove is 
$$
\bigcap_{R\in S_{1}}\drel{R}
g^{-1}\left[U\right]
=
g^{-1}\left[\bigcap_{Q\in S_{2}}\drel{Q}U\right]. 
$$

\noindent
($\subseteq$):
Take any $x\in W_{1}$ and suppose 
$$
x\in 
\bigcap_{R\in S_{1}}\drel{R}
g^{-1}\left[U\right].
$$
Then
$$
\forall R\in S_{1}
\exists w_{R}\in g^{-1}\left[U\right](x\rel{R}w_{R}).
$$
Since $g$ is a homomorphism of multi-relational Kripke frames, 
for any $Q\in S_{2}$, 
there exists $R_{Q}\in S_{1}$ such 
that for any $y\in W_{1}$
\begin{equation*}
x\rel{R_{Q}}y\thn g(x)\rel{Q}g(y). 
\end{equation*}
Therefore, for any $Q\in S_{2}$, there exists $R_{Q}\in S_{1}$ and 
$w_{R_{Q}}\in g^{-1}\left[U\right]$
such that
$$
g(x)\rel{Q}g(w_{R_{Q}}). 
$$
Hence, 
$$
g(x)\in\drel{Q}U. 
$$
Since $Q$ is arbitrary, 
$$
g(x)\in\bigcap_{Q\in S_{2}}\drel{Q}U. 
$$
Hence, 
$$
x\in g^{-1}\left[\bigcap_{Q\in S_{2}}\drel{Q}U\right]. 
$$

\noindent
($\supseteq$):
Take any $x\in W_{1}$. Then 
\begin{align*}
x\in g^{-1}\left[\bigcap_{Q\in S_{2}}\drel{Q}U\right]
&\eq
g(x)\in\bigcap_{Q\in S_{2}}\drel{Q}U\\
&\eq
\forall Q\in S_{2}\exists u_{Q}\in U
\left(g(x)\rel{Q}u_{Q}\right). 
\end{align*}
Since $g$ is a homomorphism of multi-relational 
Kripke frames, for any $R\in S_{1}$, there exists $Q_{R}\in S_{2}$ such 
that for any $u\in W_{2}$
\begin{equation*}
g(x)\rel{Q_{R}}u \thn \mbox{$\exists y\in W_{1}$ such that 
$x\rel{R}y$
and
$g(y)=u$}. 
\end{equation*}
Therefore, for any $R\in S_{1}$, there exist $Q_{R}\in S_{1}$, 
$u_{Q_{R}}\in U$, and $y\in W_{1}$ such that 
$$
x\rel{R}y,\ \ 
g(y)=u_{Q_{R}}\in U. 
$$
Hence, 
$$
x\in \drel{R} g^{-1}\left[U\right].
$$
Since $R$ is arbitrary, 
$$
x\in \bigcap_{R\in S_{1}}\drel{R} g^{-1}\left[U\right].
$$
\end{proof}

\section{Duality between $\cama_{\kappa}$ and $\mkf_{\kappa}$}

In this section, we show that for any regular cardinal $\kappa$, 
$$
\ident_{\cama_{\kappa}}\cong G\circ F,\ \ 
\ident_{\mkf_{\kappa}}\cong F\circ G.  
$$

\begin{proposition}
Let $\kappa$ be a regular cardinal. 
For any object $A$ of $\cama_{\kappa}$, define 
a mapping $\tau_{A}:A\yy G(F(A))$ by 
$$
\tau_{A}(x)=\{a\in\atom(A)\mid a\leq x\}
$$
for any $x\in A$. 
Then $\tau$ is a natural transformation from 
$\ident_{\cama_{\kappa}}$ to $G\circ F$. 
\end{proposition}

\begin{proof}
Let $f:A\yy B$ be an arrow of $\cama_{\kappa}$. 
Then for any $x\in A$ and $b\in\atom(B)$, 
\begin{align*}
b\in G(F(f))\circ\tau_{A}(x)
&\eq
b\in (\ladj{f})^{-1}\left[\{a\in\atom(A)\mid a\leq x\}\right]\\
&\eq
\ladj{f}(b)\leq x\\
&\eq
b\leq f(x)\\
&\eq
b\in\tau_{B}\circ f(x).
\end{align*}
Hence, 
$$
G(F(f))\circ\tau_{A}=\tau_{B}\circ f. 
$$
\end{proof}

\begin{theorem}\label{camatomkf}
Let $\kappa$ be a regular cardinal. 
For any object $A$ of $\cama_{\kappa}$, 
$\tau_{A}:A\yy G(F(A))$ is an isomorphism of complete modal algebras. 
\end{theorem}

\begin{proof}
It is clear that $\tau_{A}$ is an isomorphism of complete Boolean algebras. 
We show that 
$$
\tau_{A}(\mdl x)=\mdl_{F(A)}\tau_{A}(x)
$$
for any $x\in A$. What we have to show is 
$$
\{a\in\atom(A)\mid a\leq\mdl x\}
=
\bigcap_{X\subseteq A,\ |X|<\kappa}\drel{R(X)}\{a\in\atom (A)\mid a\leq x\}. 
$$

\noindent
($\subseteq$):
Suppose $a\leq\mdl x$. Take any $X\subseteq A$ such that $|X|<\kappa$. 
If $x\leq p(X,a)$, then 
\begin{align*}
a
&\leq
\mdl x\\
&\leq 
\mdl p(X,a)\\
&=
\mdl \ior\mdl^{-1}\left[\da(-a)\right]\cap X\\
&=
\ior\mdl\left[\mdl^{-1}\left[\da(-a)\right]\cap X\right]
&& \text{($\kappa$-additivity)}\\
&\leq
\ior \da(-a)\\
&=
-a,  
\end{align*}
which contradicts to $a\in\atom(A)$. Hence, $x\not\leq p(X,a)$. 
As $A$ is atomic, there exists $b\in\atom(A)$ such that 
$b\leq x$ and $b\not\leq p(X,a)$. 
Then $a\rel{R(X)}b$ by Lemma \ref{equivalence}, and
$$
a\in\drel{R(X)}\{b\in\atom(A)\mid b\leq x\}. 
$$
As $X$ is taken arbitrarily, 
$$
a\in\bigcap_{X\subseteq A,\ |X|<\kappa}\drel{R(X)}\{b\in\atom(A)\mid b\leq x\}. 
$$

\noindent
($\supseteq$):
Suppose
$a\not\leq\mdl x$. Then for any $b\in\atom(A)$ such that $b\leq x$,  
$$
a\not\leq \mdl x=\iand\mdl \left[\ua b\cap \{x\}\right]. 
$$
Hence, 
$$
a\not\in\drel{R\left(\{x\}\right)}\{b\in\atom(A)\mid b\leq x\}. 
$$
Thus, 
$$
a\not\in\bigcap_{X\subseteq A,\ |X|<\kappa}\drel{R(X)}\{b\in\atom(A)\mid b\leq x\}. 
$$
\end{proof}

\begin{proposition}
Let $\kappa$ be a regular cardinal. 
For any object $M=\langle W,S\rangle$ of $\mkf_{\kappa}$, define 
$\theta_{M}:M\yy F(G(M))$ by 
$$
\theta_{M}(w)=\{w\}
$$
for any $w\in W$. 
Then $\theta$ is a natural transformation from 
$\ident_{\mkf_{\kappa}}$ to $F\circ G$. 
\end{proposition}

\begin{proof}
For any $M$,  
$\theta_{M}$ is well-defined as a mapping,  
since 
$$\atom(G(M))=\{\{w\}\mid w\in W\}.
$$
Let $M_{1}=\langle W_{1},S_{1}\rangle$ and 
$M_{2}=\langle W_{2},S_{2}\rangle$ be objects of $\mkf_{\kappa}$, and  
$g:M_{1}\yy M_{2}$ an arrow of $\mkf_{\kappa}$. 
Then for any $w\in W_{1}$, 
\begin{align*}
F(G(g))\circ\theta_{M_{1}}(w)
&=
G(g)_{\ast}(\{w\})\\
&=
\bigcap \{X\subseteq W_{2}\mid w\in G(g)(X)\}\\
&=
\bigcap \{X\subseteq W_{2}\mid w\in g^{-1}\left[X\right]\}\\
&=
\bigcap \{X\subseteq W_{2}\mid g(w)\in X\}\\
&=
\{g(w)\}\\
&=
\theta_{M_{2}}\circ g(w).
\end{align*}
Hence, 
$$
F(G(g))\circ\theta_{M_{1}}=\theta_{M_{2}}\circ g. 
$$
\end{proof}

\begin{theorem}\label{mkftocama}
Let $\kappa$ be a regular cardinal. 
For any object $M=\langle W, S\rangle$ of $\mkf_{\kappa}$, 
$\theta_{M}:M\yy F(G(M))$ is an isomorphism of multi-relational 
Kripke frames. 
\end{theorem}

\begin{proof}
It is clear that $\theta_{M}$ is a set-theoretical bijection. 
We show that it is a homomorphism of multi-relational Kripke frames. 
By definition of $G$ and $F$, 
$$
F(G(M))
=
\langle
\left\{\{w\}\mid w\in W\right\},
\left\{R(U)\mid U\subseteq \power(W),\ |U|<\kappa\right\}
\rangle, 
$$
where 
\begin{align*}
\{w_{1}\}\rel{R(U)}\{w_{2}\}
&\eq
\{w_{1}\}\subseteq\bigcap\mdl_{M}\left[\ua\{w_{2}\}\cap U\right]. 
\end{align*}
By definition of $\mdl_{M}$ in $G(M)$, 
\begin{align*}
\{w_{1}\}\rel{R(U)}\{w_{2}\}
&\eq
\{w_{1}\}\subseteq
\bigcap\left\{\bigcap_{R\in S}\drel{R} X
\mid
X\in\ua\{w_{2}\}\cap U\right\}\\
&\eq
\forall X\in U
\left(w_{2}\in X\thn
w_{1}\in\bigcap_{R\in S}\drel{R} X
\right) \\
&\eq
\forall X\in U\left(
w_{1}\not\in\bigcap_{R\in S}\drel{R} X
\thn
w_{2}\not\in X\right). 
\end{align*}

\noindent
Condition 1 of Definition \ref{def:mkf}: 
Take any $w\in W$ and any $U\in \power(W)$ such that $|U|<\kappa$. 
For any $X\in U$, if $w\not\in\bigcap_{R\in S}\drel{R} X$,  
then we can fix one $R_{X}\in S$ such that $w\not\in\drel{R_{X}}X$. 
Since $M$ is $\kappa$-downward directed, there exists $Q\in S$ such that 
$$
Q\subseteq
\bigcap\left\{R_{X}\mid X\in U,\ w\not\in\bigcap_{R\in S}\drel{R} X\right\}. 
$$
We claim that for any $w'\in W$, 
$$
w\rel{Q}w'\thn \{w\}\rel{R(U)}\{w'\}.
$$
Suppose $w\rel{Q}w'$. 
Take any $X\in U$ and 
suppose $w\not\in\bigcap_{R\in S}\drel{R} X$.  
Then $w\not\in\drel{R_{X}}X$. 
As $w\rel{R_{X}}w'$ by definition of $Q$, 
$w'\not\in X$.

\noindent
Condition 2 of Definition \ref{def:mkf}: 
Take any $w\in W$ and any $R\in S$. 
Let 
$$
U=\left\{W\setminus\urel{R}w\right\}. 
$$
Clearly, 
$$
w\not\in\drel{R}\left(W\setminus\urel{R}w\right). 
$$
Therefore, 
$$
w\not\in\bigcap_{Q\in S}\drel{Q}\left(W\setminus\urel{R}w\right). 
$$
Hence, for any $v\in W$, 
\begin{align*}
\{w\}\rel{R(U)}\{v\} 
&\eq
v\not\in W\setminus\urel{R}w\\
&\eq
w\rel{R}v. 
\end{align*}
\end{proof}

\begin{theorem}\label{mkfcama}
For any regular cardinal $\kappa$, 
$\cama_{\kappa}$ and $\mkf_{\kappa}$ are dually equivalent. 
\end{theorem}

\begin{proof}
Theorem \ref{camatomkf} and Theorem \ref{mkftocama}. 
\end{proof}

\begin{corollary}\label{mkfmap}
Let $M_{1}=\langle W_{1},S_{1}\rangle$ and $M_{2}=\langle W_{2},S_{2}\rangle$ be 
multi-relational Kripke frames. 
A mapping $f:W_{1}\yy W_{2}$ is a homomorphism of multi-relational Kripke frames 
from $M_{1}$ to $M_{2}$
if and only if the mapping $g:\power(M_{2})\yy \power(M_{1})$ which is defined by 
$$
g:S\mapsto f^{-1}[S]
$$
for any $S\subseteq W_{2}$ 
is a homomorphism of complete modal 
algebras from $G(M_{2})$ to $G(M_{1})$.
\end{corollary}

\begin{proof}
We only show the if-part. 
Suppose that $g$ is a homomorphism of complete modal 
algebras.
Then
$F(g):FG(M_{2})\yy FG(M_{2})$ is a homomorphism of multi-relational Kripke frames. 
Let 
$$
h=\theta_{M_{2}}^{-1}\circ F(g)\circ \theta_{M_{1}}. 
$$
By definition of $\theta$ and $\tau$, 
the composite of $G\theta$ and $\tau_{G}$ is the identity natural transformation on $G$. 
Hence, for any $S\subseteq\power(W_{2})$, 
\begin{align*}
h^{-1}[S]
&=
G(h)(S)\\
&=
G(\theta_{M_{1}})\circ GF(g)\circ G(\theta_{M_{2}}^{-1})\\
&=
\tau_{G(M_{1})}^{-1}\circ GF(g)\circ \tau_{G(M_{2})}\\
&=
g(S)\\
&=
f^{-1}[S]. 
\end{align*}
Thus, $f=h$ is a homomorphism of multi-relational Kripke frames.

$$
\xymatrix{
M_{1} \ar[d]^{\theta_{M_{1}}} \ar[r]^{h} & M_{2} \ar[d]^{\theta_{M_{2}}}\\
FG(M_{1})\ar[r]^{F(g)} & FG(M_{2})
}
\hspace{20pt}
\xymatrix{
G(M_{1})  & G(M_{2}) \ar[l]^{G(h)} \\
GFG(M_{1})\ar[u]_{G\theta_{M_{1}}}& GFG(M_{2})\ar[l]^{GF(g)} \ar[u]_{G\theta_{M_{2}}}\\
G(M_{1}) \ar[u]_{\tau_{G(M_{1})}} & G(M_{2})\ar[l]^{g} \ar[u]_{\tau_{G(M_{2})}}
}
$$
\end{proof}

\section{Application}

As an application of the duality theorem, 
we show that 
for any regular cardinals 
$\kappa$ and $\kappa'$ with $\kappa<\kappa'$, 
the inclusion functor from 
$\cama_{\kappa'}$ to $\cama_{\kappa}$ 
and 
that
from $\mkf_{\kappa'}$ to $\mkf_{\kappa}$ 
are not essentially surjective,   
where
a functor $F$ from a category $C$ to a category $D$ is said to be 
{\em essentially surjective}, if for any object $d$ of $D$, 
there exists an object $c$ of $C$ such that $F(c)$ is isomorphic to $d$.

The following proposition is based on Fact 4.5 of \cite{mnr16}.

\begin{proposition}\label{minari}
Let $\kappa$ and $\kappa'$ be regular cardinals.  
If $\kappa<\kappa'$, 
there exists a complete atomic modal algebra $A$ which is 
$\kappa$-additive but not $\kappa'$-additive. 
\end{proposition}

\begin{proof}
Consider a multi-relational Kripke frame $M$ defined by 
$$
M=
\left\langle
\kappa\cup\{\infty\},
\{Q_{X}\mid X\subseteq\kappa,\ |X|<\kappa\}
\right\rangle
$$
where 
$$
Q_{X}=\{(\infty,\alpha)\mid \alpha\not\in X\}. 
$$
Suppose $|I|<\kappa$, and for any $i\in I$, suppose 
$X_{i}\subseteq \kappa$ and $|X_{i}|<\kappa$. 
Then $|\bigcup_{i\in I}X_{i}|<\kappa$ and 
$$
Q_{\bigcup_{i\in I}X_{i}}
=
\bigcap_{i\in I}Q_{X_{i}}. 
$$
Hence, $M$ is an object of $\mkf_{\kappa}$. Therefore, by the duality theorem, 
$G(M)$ is an object of $\cama_{\kappa}$.  
We show that in $G(M)$,   
$$
\mdl_{M}\ior_{i\in\kappa}\{i\}
\not\leq
\ior_{i\in\kappa}\mdl_{M}\{i\}. 
$$
For any
$X\subseteq\kappa$ such that $|X|<\kappa$, 
there exists $i\in \kappa$ such that 
$i\not\in X$. Hence, 
$$
\infty\in\bigcap_{X\subseteq\kappa,\ |X|<\kappa}\drel{Q_{X}}
\bigcup_{i\in\kappa}
\{i\}. 
$$
Thus, 
$$
\infty\in
\mdl_{M}\ior_{i\in\kappa}\{i\}.
$$

\noindent
On the other side, 
for any $i\in\kappa$, 
$$
\infty\not\in
\drel{Q_{\{i\}}}\{i\}. 
$$
Therefore,
$$
\infty \not\in
\bigcap_{X\subseteq\kappa,\ |X|<\kappa}\drel{Q_{X}}
\{i\}. 
$$
Since $i$ is taken arbitrarily
$$
\infty \not\in
\bigcup_{i\in I}
\bigcap_{X\subseteq\kappa,\ |X|<\kappa}\drel{Q_{X}}
\{i\}.
$$
Hence, 
$$
\infty\not\in
\ior_{i\in\kappa}\mdl_{M}\{i\}. 
$$
\end{proof}

\begin{theorem}
Let $\kappa$ and $\kappa'$ be regular cardinals such that $\kappa<\kappa'$.  
Then 
the inclusion functor from 
$\cama_{\kappa'}$ to $\cama_{\kappa}$ 
and that
from $\mkf_{\kappa'}$ to $\mkf_{\kappa}$ 
are not essentially surjective.  
\end{theorem}

\begin{proof}
Let $M$  be the multi-relational Kripke frame defined in Proposition \ref{minari}. 
Then 
$G(M)$ is an object of $\cama_{\kappa}$, and it is clear that 
no objects of $\cama_{\kappa'}$ are isomorphic to $G(M)$. 
Hence, by Theorem \ref{mkfcama}, 
no objects of $\mkf_{\kappa'}$ are isomorphic to $M$. 
\end{proof}


\newcommand{\noop}[1]{}

\end{document}